\documentclass[12pt,reqno]{amsart}

\usepackage{etex}
\usepackage[utf8]{inputenc}
\usepackage{amsfonts}
\usepackage{amsmath}
\usepackage{amssymb}
\usepackage{amsthm}
\usepackage{mathrsfs}
\usepackage{stmaryrd}
\usepackage{color}
\usepackage[english]{babel}
\usepackage[T1]{fontenc}
\usepackage{url}
\usepackage{graphicx}
\usepackage[pdftex,colorlinks,backref=page,citecolor=blue]{hyperref}
\usepackage{caption}
\usepackage{enumerate}
\usepackage{epstopdf}
\usepackage{hyphenat}
\usepackage{float}
\usepackage{indentfirst}
\usepackage[export]{adjustbox}
\usepackage{tikz}
\usepackage{color}
\usepackage[all]{xy}
\usetikzlibrary{matrix,arrows,decorations.pathmorphing}
\usepackage{booktabs}
\usepackage{array}
\usepackage{bm}
\usepackage{csquotes}
\usepackage{todonotes}

\newcolumntype{P}[1]{>{\centering\arraybackslash}p{#1}}

\allowdisplaybreaks

\newtheorem{theorem}{Theorem}[section]

\newtheorem{lemma}[theorem]{Lemma}
\newtheorem{proposition}[theorem]{Proposition}

\newtheorem*{theoremx}{Theorem}

\theoremstyle{definition}

\newtheorem{remark}[theorem]{Remark}

\newtheorem{example}[theorem]{Example}

\newcommand{\kk}{\Bbbk}

\makeatletter
\def\l@subsection{\@tocline{2}{0pt}{2.5pc}{5pc}{}}
\makeatother

\author{Alessio Caminata}
\address{Dipartimento di Matematica, Dipartimento di Eccellenza 2023-2027, Universit\`a di Genova, via Dodecaneso 35, 16146, Genova, Italy}
\email{alessio.caminata@unige.it}

\author{Enrico Carlini}
\address{DISMA-Dipartimento di Scienze Matematiche, Politenico di Torino, Corso Duca degli Abruzzi 24, 10129, Torino}
\email{enrico.carlini@polito.it}

\author{Luca Schaffler}
\address{Dipartimento di Matematica e Fisica, Universit\`a degli Studi Roma Tre, Largo San Leonardo Murialdo 1, 00146, Roma, Italy}
\email{luca.schaffler@uniroma3.it}

\title[Simplices osculating rational normal curves]{Simplices osculating rational normal curves}


\begin{document}
	
	\maketitle
	
\begin{abstract}
A classical result of von Staudt states that if eight planes osculate a twisted cubic curve and we divide them into two groups of four, then the eight vertices of the corresponding tetrahedra lie on a twisted cubic curve. In the current paper, we give an alternative proof of this result using modern tools, and at the same time we prove the analogous result for rational normal curves in any projective space. This higher dimensional generalization was claimed without proof in a paper of H.S. White in 1921.
\end{abstract}

\section{Introduction}

Studying the existence of rational normal curves satisfying a given set of conditions is a classical and widely studied problem. The basic instance of this is Castelnuovo's lemma: $d+3$ linearly general points in $\mathbb{P}^d$ lie on a unique rational normal curve. Replacing the points with $d+3$ codimension two linear spaces gives rise to a richer situation, existence and non-existence conditions are discussed in \cite{CarliniCatalisanoExistence}. Moreover, the existence of suitable rational normal curves can be used to successfully address postulation problems  (see e.g. \cite{CarliniCatalisanoExistence,CarliniCatalisanoONrnc}).

When more than $d+3$ points are considered, the existence problem becomes more subtle because it involves points in special position. The first instance of this situation is Pascal's theorem \cite{Pas40}, which characterizes the choices of six points on the plane lying on a conic. In relation to this, but from a different perspective, we have a theorem of Brianchon \cite[p.~35]{Bri17}, which asserts that if we take six tangent lines to a plane conic and we divide them into two groups of three, then the six vertices of the corresponding triangles also lie on a conic.  This is pictured in Figure~\ref{fig:Brianchon-result}. An alternative proof of Brianchon's theorem can be found in \cite{CarliniCatalisanoFavacchioGuardo}.

\begin{figure}
\begin{tikzpicture}[scale=0.55]
\draw [rotate around={0.:(0.,0.)},line width=1.5pt] (0.,0.) ellipse (5.cm and 4.cm);
\draw [rotate around={0.:(0.,0.)},line width=1.5pt,dash pattern=on 4pt off 4pt] (0.,0.) ellipse (8.333333333333334cm and 2.5cm);
\draw [line width=1pt] (-8.333333333333334,0.) circle (0.15cm);
\draw [line width=1pt] (-5.,2.) circle (0.15cm);
\draw [line width=1pt] (-5.,-2.) circle (0.15cm);
\draw [line width=1pt] (5.,-2.) circle (0.15cm);
\draw [line width=1pt] (8.333333333333334,0.) circle (0.15cm);
\draw [line width=1pt] (5.,2.) circle (0.15cm);
\draw (-4,4.4) node[anchor=north west] {$P_1$};
\draw (-6.4,0.5) node[anchor=north west] {$P_2$};
\draw (-4,-3.1) node[anchor=north west] {$P_3$};
\draw (3,-3.1) node[anchor=north west] {$P_1'$};
\draw (5,0.5) node[anchor=north west] {$P_2'$};
\draw (3,4.4) node[anchor=north west] {$P_3'$};
\draw (-6.3,3.0) node[anchor=north west] {$R_1$};
\draw (-9.2,-0.2) node[anchor=north west] {$R_2$};
\draw (-6.2,-2) node[anchor=north west] {$R_3$};
\draw (5,-2) node[anchor=north west] {$R_1'$};
\draw (7.9,-0.2) node[anchor=north west] {$R_2'$};
\draw (5.0,3.0) node[anchor=north west] {$R_3'$};
\draw [line width=1pt] (-0.8574929257125442,5.514495755427527)-- (9.190826259045885,-0.5144957554275305);
\draw [line width=1pt] (0.8574929257125441,5.514495755427527)-- (-9.190826259045885,-0.5144957554275305);
\draw [line width=1pt] (-9.190826259045885,0.5144957554275305)-- (0.8574929257125442,-5.514495755427527);
\draw [line width=1pt] (-0.8574929257125441,-5.514495755427527)-- (9.190826259045885,0.5144957554275305);
\draw [line width=1pt] (5.,-3.)-- (5.,3.);
\draw [line width=1pt] (-5.,3.)-- (-5.,-3.);
\begin{scriptsize}
\draw [fill=black] (5.,0.) circle (3.5pt);
\draw [fill=black] (-3.,3.2) circle (3.5pt);
\draw [fill=black] (-5.,0.) circle (3.5pt);
\draw [fill=black] (-3.,-3.2) circle (3.5pt);
\draw [fill=black] (3.,-3.2) circle (3.5pt);
\draw [fill=black] (3.,3.2) circle (3.5pt);
\end{scriptsize}
\end{tikzpicture}
\caption{Brianchon's theorem on six points on a smooth conic.}
\label{fig:Brianchon-result}
\end{figure}
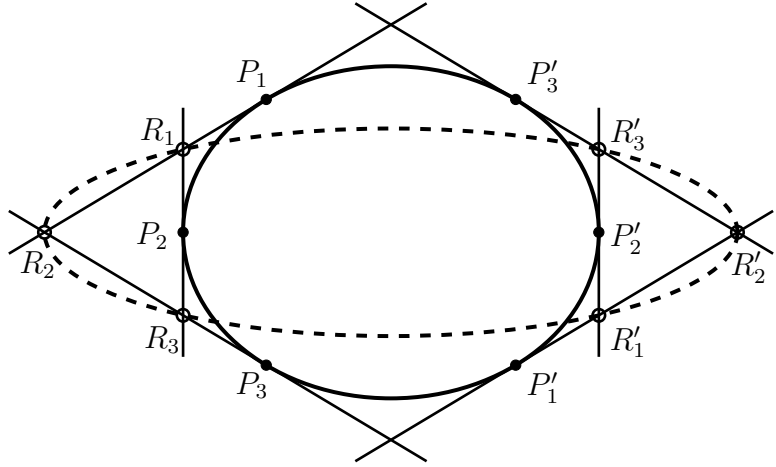

This result was generalized by von Staudt \cite[Satz~588, p.~377]{Sta56} to twisted cubic curves as follows. Consider eight osculating planes to a twisted cubic curve and divide them into two groups of four. Then, the eight vertices of the corresponding two tetrahedra lie on a twisted cubic curve. The idea of the proof is that the two tetrahedra are polar with respect to an appropriately constructed quadric. Hurwitz \cite{Hur82} gave an alternative geometric proof of von Staudt's result by providing a synthetic construction of the second twisted cubic curve. This was achieved using suitable pencils of planes and applying the fact that a degree $3$ space curve intersects a plane at three points. H.S. White \cite{Whi21} gave an alternative algebraic proof. The argument also used polarity as von Staudt, but the approach is based on algebraic coordinates rather than synthetic geometric arguments. In \cite{Whi21} it is also stated that the theorem is \textit{\enquote{extensible at once to norm-curves in any number of dimension.}} While the statement can be naturally generalized to rational normal curves of degree $d$ in projective space $\mathbb{P}^d$ for any $d\geq 3$ (see Theorem~\ref{thm:gwt}), a proof is not provided in the above references. Moreover, even for the case $d=3$ the proofs in \cite{Hur82, Sta56, Whi21} are written in a classical mathematical language that is not easily accessible nowadays.
In \cite[\S\,4]{Whi27} F.P. White also studied the geometry of osculating simplices to arbitrary rational normal curves, but in a different direction. In the current paper, our goal is to prove Brianchon's and von Staudt's theorems and their higher dimensional generalization, stated below, in modern terms.

\begin{theoremx}[Theorem~\ref{thm:gwt}]
Let $d\geq2$ and let $\mathcal{C}\subseteq\mathbb{P}^d$ be a rational normal curve. Consider distinct points $P_i,P_i'\in \mathcal{C}$ for $1\leq i\leq d+1$ and let $\pi_i$, respectively $\pi_i'$, be the osculating hyperplanes to $\mathcal{C}$ at $P_i$, respectively $P_i'$. Define
\[
R_i=\bigcap_{j\neq i} \pi_j \mbox{ and }R_i'=\bigcap_{j\neq i} \pi_j'.
\]
Then, the points $R_1,\ldots,R_{d+1},R_1',\ldots,R_{d+1}'$ lie on a rational normal curve.
\end{theoremx}

This recovers Brianchon's theorem when $d=2$ and von Staudt's result when $d=3$. In the latter case, we remark that the above theorem is rather the projective dual of the statement in \cite{Sta56,Hur82,Whi21}, which is more convenient for our purposes.

The first main idea behind the proof is to consider $V_{d,n}\subseteq(\mathbb{P}^d)^n$, which is the Veronese compactification of the parameter space of $n$-point configurations in $\mathbb{P}^d$ that can lie on a rational normal curve. In \cite{CGMS18} the authors constructed a set of multi-homogeneous equations defining a variety $W_{d,n}\subseteq(\mathbb{P}^d)^n$ containing $V_{d,n}$. We recall these equations in \S\,\ref{sec:Vdn-and-equations}. The variety $W_{d,n}$ coincides with $V_{d,n}$ in some cases, but it is in general a larger space. However, we prove in Proposition~\ref{prop:pts-glp-rnc-Wdn} that for linearly general point configurations, the equations defining $W_{d,n}$ characterize when the configuration lies on a rational normal curve. This is the first ingredient of the proof of our main result. As the points $R_i,R_i'$, $i=1,\ldots,d+1$, are linearly general by Lemma~\ref{prop:tetrahedrons-glp}, our main goal becomes to show that the points $R_i,R_i'$ satisfy the equations of $W_{d,2d+2}$. This is a subtle algebro-combinatorial argument which we carry out in \S\,\ref{sec:proof-gen-vonStaudt}.

We work over an algebraically closed field $\kk$ of characteristic $0$ or larger than $d$.
	
	
	\section{Points on rational normal curves}
	
	Let $d,n$ be positive integers satisfying $n\geq d+4$ and $d\geq2$.
	
	
	\subsection{The space $V_{d,n}$ and its defining equations}
	\label{sec:Vdn-and-equations}
	
	A \emph{rational normal curve} in $\mathbb{P}^d$ is a projective degree $d$ irreducible smooth rational curve. Let us recall the definition of $V_{d,n}$ \cite[Definition~2.1]{CGMS18}, which is the Zariski closure in $(\mathbb{P}^d)^n$ of the following set:
	\[
	\{(p_1,\ldots,p_n)\in(\mathbb{P}^d)^n\mid p_1,\ldots,p_n~\textrm{are distinct and lie on a rational normal curve}\}.
	\]
	So, $V_{d,n}$ generically parametrizes $n$-tuples of distinct points in $\mathbb{P}^d$ that lie on a rational normal curve. The point configurations parametrized by the boundary of $V_{d,n}$ are either degenerate, i.e. contained in a hyperplane, or they lie on specific flat degenerations of rational normal curves called \emph{quasi-Veronese curves} (points are allowed to collide). By a result of Artin \cite[Lemma~13.1]{Art76}, a quasi-Veronese curve $\mathcal{C}$ is characterized by three conditions: (1) each irreducible component of $\mathcal{C}$ is a rational normal curve in its span, (2) the singularities of $\mathcal{C}$ are \'etale locally the union of coordinate axes in $\kk^m$, and
	(3) any connected closed subcurve of $\mathcal{C}$ is again a quasi-Veronese curve in its span.
	
	It is natural to look for a set of multihomogeneous equations cutting out $V_{d,n}$. This problem was studied in \cite{CGMS18}, where the following equations were considered. Let $[n]:=\{1,\ldots,n\}$ and, for a finite set $X$, denote by $\binom{X}{k}$ the set of $k$-element subsets of $X$. For $J\in\binom{[n]}{d+4}$, $I=\{i_1<\ldots<i_6\}\in\binom{J}{6}$, $J\setminus I=\{j_1<\ldots<j_{d-2}\}$, we define
	\begin{align*}
		\psi_{I,J}:=&|i_4i_5i_6j_1\ldots j_{d-2}||i_2i_3i_6j_1\ldots j_{d-2}||i_1i_3i_5j_1\ldots j_{d-2}||i_1i_2i_4j_1\ldots j_{d-2}|\\
		-&|i_3i_5i_6j_1\ldots j_{d-2}||i_2i_4i_6j_1\ldots j_{d-2}||i_1i_4i_5j_1\ldots j_{d-2}||i_1i_2i_3j_1\ldots j_{d-2}|,
	\end{align*}
	where $|i_4i_5i_6j_1\ldots j_{d-2}|$ denotes the determinant of the $(d+1)\times(d+1)$-matrix whose columns are the coordinates of the points $p_{i_4},p_{i_5},p_{i_6},p_{j_1},\ldots,p_{j_{d-2}}$, and so on. It is customary to refer to these determinants as \emph{brackets} \cite[\S\,3]{Stu08} and to think of $\psi_{I,J}$ as the difference of two \emph{bracket monomials}. Note that the vanishing of $\psi_{I,J}$ at a given point configuration is independent from the choice of the projective coordinates of the points because each index $i_h$ appears exactly twice in each bracket monomial and each $j_h$ appears in all the brackets.

We let $W_{d,n}$ be the subscheme of $(\mathbb{P}^d)^n$ defined by the equations $\psi_{I,J}$ as $I,J$ vary. We also introduce $Y_{d,n}\subseteq(\mathbb{P}^d)^n$ as the subscheme cut out by the maximal minors of the $(d+1)\times n$ matrix of variables of $(\mathbb{P}^d)^n$. Equivalently, $Y_{d,n}$ parametrizes $n$-point configurations in $\mathbb{P}^d$ which lie on a hyperplane. We have that $Y_{d,n}\subseteq W_{d,n}$. The next theorem summarizes Theorem~3.6, Theorem~4.19, Theorem~4.22, and Proposition~4.25 from \cite{CGMS18}.
	
	\begin{theorem}[{\cite{CGMS18}}]\label{thm:equationsVdn}
		The following set-theoretic equalities hold:
		\begin{enumerate}
			
			\item $W_{d,d+4}=V_{d,d+4}\cup Y_{d,d+4}$ for all $d\geq2$;
			
			\item $W_{3,n}=V_{3,n}\cup Y_{3,n}$ for all $n\geq7$;
			
			\item $W_{d,n}=V_{d,n}$ if and only if $d=2$ and $n\geq6$, or $(d,n)=(3,7),(3,8),(4,8)$.
			
		\end{enumerate}
	\end{theorem}
	
	\begin{remark}
		Notice that the equations defining $W_{d,n}$ in \cite{CGMS18} differ by a sign with respect to the ones we adopt in this section and in the rest of the paper. This was discussed in \cite[\S\,4]{CS21}
	\end{remark}

\begin{remark}	
One can ask whether the equality $W_{d,n}=V_{d,n}\cup Y_{d,n}$ holds, at least set-theoretically. In \cite{CGMS18} it is conjectured that this is the case. We observe that the containment $W_{d,n}\supseteq V_{d,n}\cup Y_{d,n}$ is true scheme-theoretically by \cite[Lemma~4.17 and Lemma~4.18]{CGMS18}.
\end{remark}
	
	\begin{example}
		Let us study the specific case of $d=3$ and $n=8$. In this case we have $56$ equations defining $W_{3,8}=V_{3,8}$. We can construct the equations in the following way. First, choose $J\in\binom{[8]}{7}$, say $J=\{1,\ldots,7\}$. Then, choose $I\in\binom{[J]}{6}$, say $I=\{1,\ldots,6\}$. Then,
		\[
		\psi_{I,J}= |4567||2367||1357||1247| - |3567||2467||1457||1237|=0.
		\]
	\end{example}
	
	
\subsection{Osculating hyperplanes to rational normal curves}
\label{sec:osculatingplanes}

 We briefly summarize some basic facts about the intrinsic definition of rational normal curves, apolarity, and osculating hyperplanes. The interested reader can refer to \cite{RS2011} and \cite{GeramitaQUEENS} for a more thorough analysis.
	
Let $S=\kk[x_0,x_1]$ and denote by $S_d=\kk[x_0,x_1]_d$ the degree $d$ homogeneous piece of $S$. Consider the map
\[
\iota\colon\mathbb{P}(S_1) \longrightarrow \mathbb{P}(S_d)
\]
defined as $\iota([L])=[L^d]$. The image of $\iota$ is a rational normal curve $\mathcal{C}$ of degree $d$. The osculating hyperplane to $\mathcal{C}$ in a point $P=[L^d]$ is the unique hyperplane intersecting set theoretically $\mathcal{C}$ only in $P$.

To describe the hyperplanes in $\mathbb{P}(S_d)$ we use duality via apolarity theory. Namely, consider the ring $T=\kk[\partial_0,\partial_1]$, where $\partial_i$ is the partial derivative with respect to $x_i$. Consider the natural action of $T_d$ on $S_d$ given by
\begin{align*}
T_d\times S_d&\rightarrow\kk\\
(\partial,F)&\mapsto\partial\circ F,
\end{align*}
where $\partial\circ F=\partial F$ denotes the formal partial derivative $\partial$ (of order $d$) applied to $F$.
The above is called \emph{apolarity action}, and is a perfect pairing, so we can identify the dual of $\mathbb{P}(S_d)$ with $\mathbb{P}(T_d)$ (here we use the hypothesis that the characteristic is $0$ or larger than $d$). More explicitly, a hyperplane in $\mathbb{P}(S_d)$ is the kernel of an element in $\mathbb{P}(T_d)$.
	
Consider a linear form $L=ax_0+bx_1$ and define the corresponding operator $\partial_L=b\partial_0-a\partial_1$. Clearly, $\partial_L\circ L=0$. Hence, the hyperplane defined by $\partial_L^d$ contains the point $P=[L^d]$ and is made of points of the form $[LG]$ for $G\in S_{d-1}$. The hyperplane defined by $\partial_L^d$ only intersects $\mathcal{C}$ in $P=[L^d]$ and thus it is the osculating hyperplane to the rational normal curve at the point.

Another useful approach to osculating hyperplanes is via coordinates. Fix bases so that, $\iota\colon\mathbb{P}^1\hookrightarrow\mathbb{P}^d$ is described as follows:
	\begin{equation}
		\label{eq:standard-rational-normal-curve-for-us}
		\iota([a:b])=\left[\ldots:\binom{d}{i}a^{d-i}b^i:\ldots\right].
	\end{equation}
We call the image of $\iota$ the \emph{standard rational normal curve}. A hyperplane $H(X_0,\ldots,X_d)=0$ is an osculating hyperplane to the standard rational normal curve at the point $\iota([a_0:b_0])$ if and only if the polynomial $H(\iota([a:b]))$ has a root of order $d$ in $[a_0:b_0]$, that is, $H(\iota([a:b]))=c(a_0 b-b_0 a)^d$ for some nonzero constant $c$. After comparing the coefficients of $H(\iota([a:b]))$ with the coefficients of the expansion of $(a_0 b-b_0 a)^d$ we obtain that the hyperplane $H(X_0,\ldots,X_d)$ is given by
	\[
	b_0^dX_0-a_0b_0^{d-1}X_1+\ldots+(-1)^da_0^dX_d=0.
	\]
	
	\begin{remark}
		Often in the literature the standard rational normal curve does not have the binomial coefficients which appear in \eqref{eq:standard-rational-normal-curve-for-us}. We keep them to have a simpler equation for the osculating hyperplane.
	\end{remark}
	
	The osculating hyperplanes to a rational normal curve correspond to points in the dual $\mathbb{P}^{d*}$. By Piene \cite{Piene77}, we have that these points lie on a rational normal curve, and in particular they are in general linear position. Therefore, the original osculating hyperplanes also are linearly general. It is useful to summarize this in the following lemma.
	
	\begin{lemma}
		\label{lemma:inter-osc-planes=empty}
		Let $\mathcal{C}\subseteq\mathbb{P}^d$ be a rational normal curve and let $\pi_1,\dots,\pi_{d+1}$ be the osculating hyperplanes to $\mathcal{C}$ at $d+1$ distinct points. Then, the hyperplanes $\pi_1,\dots,\pi_{d+1}$ are in general linear position, in particular they have empty intersection.
	\end{lemma}

So, the intersection of $d$ distinct osculating hyperplanes to a rational normal curve in $\mathbb{P}^d$ is a point. In the next lemma we give explicit coordinates for this intersection.

\begin{lemma}
\label{lemma:inter-osc-planes=pt}
Let $\iota\colon\mathbb{P}^1\hookrightarrow\mathbb{P}^d$ be the standard rational normal curve $\mathcal{C}$ as in \eqref{eq:standard-rational-normal-curve-for-us}.		
Let $P_i=\iota([a_i:b_i])$, $i\in\{1,\ldots,d\}$ be distinct points. Let $\pi_i$ be the osculating hyperplane to $\mathcal{C}$ at $P_i$. Then, the hyperplanes $\pi_1,\ldots,\pi_d$ intersect in a point $R_{1\ldots d}=[r_0:\ldots:r_d]$, where	
		\[
		r_k=\sum_{I\in\binom{[d]}{d-k}}a_Ib_{[d]\setminus I} \ \ \ \ \ \text{ for any } k\in\{0,\ldots,d\},
		\]
	with the notational convention that for $I=\{i_1,\ldots,i_h\}\subseteq[d]$, $a_I:=a_{i_1}\ldots a_{i_h}$.
\end{lemma}

\begin{proof}
We can assume, up to composing with an automorphism of $\mathbb{P}^1$, that $b_1,\ldots,b_d$ are non-zero. Hence, up to scaling the coordinates of the points in $\mathbb{P}^1$, we can assume that these are $[a_i:1]$. From the discussion above, the osculating plane $\pi_i$ to $P_i$ in Cartesian form is given by:
		\[
		X_0-a_iX_1+\ldots+(-1)^da_i^dX_d=0.
		\]
		So, we need to check that the point $R_{1\ldots d}$ in the statement satisfies the above equations for $i\in\{1,\ldots,d\}$. By substitution, we have that
		\[
		r_0-a_ir_1+\ldots+(-1)^da_i^dr_d=\sum_{k=0}^d(-1)^ka_i^k\sum_{I\in\binom{[d]}{d-k}}a_I=\prod_{j=1}^d(-a_i+a_j)=0.\qedhere
		\]
	\end{proof}

	
\section{Generalized von Staudt's theorem: statement and combinatorial reduction}

A natural way to generalize von Staudt's theorem for two sets of four points on a twisted cubic to rational normal curves is the following. 

\begin{theorem}
\label{thm:gwt}
Let $\mathcal{C}\subseteq\mathbb{P}^d$ be a rational normal curve, consider distinct points $P_i,P_i'\in \mathcal{C}$ for $1\leq i\leq d+1$ and let $\pi_i$, respectively $\pi_i'$, be the osculating hyperplanes to $\mathcal{C}$ in $P_i$, respectively $P_i'$. Define
\[
R_i=\bigcap_{j\neq i} \pi_j \mbox{ and }R_i'=\bigcap_{j\neq i} \pi_j'.
\]
Then, the points $R_1,\ldots,R_{d+1},R_1',\ldots,R_{d+1}'$ lie on a rational normal curve.
\end{theorem}

\begin{remark}
In the above statement, the osculating hyperplanes $\pi_1,\ldots,\pi_{d+1}$ are in general linear position by Lemma~\ref{lemma:inter-osc-planes=empty}. So, any $d$ among them intersect precisely at one point, defining in this way the points $R_1,\ldots,R_{d+1}$. The same considerations hold for $\pi_1',\ldots,\pi_{d+1}'$.
Therefore, we obtain two simplices: one with vertices $R_1,\ldots,R_{d+1}$ and facets $\pi_1,\ldots,\pi_{d+1}$, the other with vertices $R_1',\ldots,R_{d+1}'$ and facets $\pi_1',\ldots,\pi_{d+1}'$. These are the higher dimensional analogue of the two tetrahedra in von Staudt's theorem.   
\end{remark}

For the rest of the paper, our goal will be to prove Theorem~\ref{thm:gwt}. The current section contains the proof of a technical proposition which is needed for this purpose.

\begin{proposition}
\label{prop:gwt+W}
We retain the set-up of Theorem~\ref{thm:gwt}. If the $2d+2$ points $R_i,R_i'$, $i=1,\ldots,d+1$ satisfy the equations of $W_{d,2d+2}$, then the points lie on a rational normal curve.
\end{proposition}

\begin{proof}
We divide the proof into two steps:
\begin{itemize}

\item[(1)] The points $R_i,R_i'$, $i=1,\ldots,d+1$ are in general linear position. This is the content of Lemma~\ref{prop:tetrahedrons-glp}.

\item[(2)] Linearly general points that satisfy the equations of $W_{d,n}$, $n\geq d+4$, lie on a rational normal curve. This is discussed in Lemma~\ref{prop:pts-glp-rnc-Wdn}.

\end{itemize}
The conclusion then follows.
\end{proof}

\begin{lemma}
\label{prop:tetrahedrons-glp}
We retain the set-up of Theorem~\ref{thm:gwt}. The points $R_1,\ldots,R_{d+1},R_1',\ldots,R_{d+1}'$ are in general linear position.
\end{lemma}

\begin{proof}
Using the intrinsic description of the rational normal curve $\mathcal{C}$ in \S\,\ref{sec:osculatingplanes} we have that $P_i=[L_i^d]$ and $P_i'=[N_i^d]$ for some distinct degree one forms $L_i,N_i\in S_1$. Moreover, from the description of the osculating hyperplane to a rational normal curve, we have that
\[
R_i=\left[\prod_{j\neq i} L_j\right] \mbox{ and } R_i'=\left[\prod_{j\neq i} N_j\right]
\]
for $1\leq i\leq d+1$. More in detail, recall that the osculating hyperplane $\pi_j$ at $P_j$ consists of points of the form $[L_jG]$ for $G\in S_{d-1}$. So, the point $R_i=\cap_{j\neq i}\pi_j$ corresponds to the claimed form. An analogous argument applies to $R_i'$.

We let $\mathcal{R}=\{R_1,\ldots,R_{d+1}\}$ and $\mathcal{R}'=\{R_1',\ldots,R_{d+1}'\}$. To prove that $\mathcal{R}\cup\mathcal{R}'$ is a set of points in general linear position, we need to show that, for each subset $\mathcal{X}\subseteq\mathcal{R}\cup\mathcal{R}'$, its linear span has maximal dimension, that is,
\[
\dim\langle\mathcal{X}\rangle=\min\lbrace d+1,|\mathcal{X}|\rbrace-1.
\]
To do so, define $\mathcal{X}_1=\mathcal{R}\cap\mathcal{X}$ and $\mathcal{X}_2=\mathcal{R}'\cap\mathcal{X}$. First, we deal with the case $\mathcal{X}_2=\emptyset$. So, $\mathcal{X}=\mathcal{X}_1\subseteq\mathcal{R}$. Up to relabeling, we may assume that
\[
\mathcal{X}=\lbrace R_1\,\ldots,R_\ell\rbrace
\]
and we want to show that $\dim\langle\mathcal{X}\rangle=\ell-1$. Assume by contradiction that $\dim\langle\mathcal{X}\rangle<\ell-1$. Without loss of generality, say that
\begin{equation}
\label{eq:P0inspanP1-Pr}
R_1=\left[\prod_{j\neq 1} L_j\right]\in\langle R_2,\ldots,R_\ell\rangle.
\end{equation}
We observe that for all $i=2,\ldots,\ell$, the polynomial
\[
\prod_{j\not\in\{2,\ldots,\ell\}} L_j~\textrm{divides}~\prod_{j\neq i}L_j.
\]
So, if $\prod_{j\neq 1} L_j$ is a linear combination of $\prod_{j\neq i}L_j$ for $i=2,\ldots,\ell$ by \eqref{eq:P0inspanP1-Pr}, then $L_1$ divides $\prod_{j\neq1} L_j$, which is a contradiction. This completes the proof of the case $\mathcal{X}_2=\emptyset$. The case $\mathcal{X}_1=\emptyset$ is analogous.
	
Now, we deal with the case $\mathcal{X}_1,\mathcal{X}_2\neq\emptyset$. We set $|\mathcal{X}_1|=\ell$ and $|\mathcal{X}_2|=s$. Up to relabeling, we may assume that
\[
\mathcal{X}_1=\{ R_1,\ldots,R_\ell\},~\mathcal{X}_2=\{ R_1',\ldots,R_s'\}.
\]
By the discussion above, we know that $\dim\langle\mathcal{X}_1\rangle=\ell-1\mbox{ and }\dim\langle\mathcal{X}_2\rangle=s-1$. First, note that the conclusion for $\ell+s-1\leq d$ implies the result when $\ell+s-1>d$. So, we focus on the former and we show that $\dim\langle\mathcal{X}\rangle=\ell+s-1$. For this purpose, it is enough to show that
	\[
	\langle\mathcal{X}_1\rangle\cap\langle\mathcal{X}_2\rangle=\emptyset.
	\]
	If $[F]\in \langle\mathcal{X}_1\rangle\cap\langle\mathcal{X}_2\rangle$, then the forms
	\[
	\prod_{j\not\in\lbrace1,\ldots,\ell\rbrace}L_j~\textrm{and}~\prod_{j\not\in\lbrace1,\ldots,s\rbrace}N_j
	\]
	divide $F$. Thus, since $L_i\neq N_j$ for all $i$ and $j$, we obtain that $F$ has at least
	\[
	(d+1-\ell)+(d+1-s)=d+1+d-(\ell+s-1)\geq d+1
	\]
	factors, which is a contradiction. Hence, $\langle\mathcal{X}_1\rangle\cap\langle\mathcal{X}_2\rangle=\emptyset$ and the result is now proved.
\end{proof}

\begin{lemma}
\label{prop:pts-glp-rnc-Wdn}
Let $n\geq d+4$ and consider $\mathbf{p}=(p_1,\ldots,p_n)\in W_{d,n}$ such that $p_1,\ldots,p_n$ are in general linear position. Then $p_1,\ldots,p_n$ lie on a rational normal curve.
\end{lemma}

\begin{proof}
We prove that: (1) $\mathbf{p}\in V_{d,n}$, so $p_1,\ldots,p_n$ lie on a quasi-Veronese curve $C$; (2) $C$ is a smooth rational curve.

To prove (1), we show that $\mathbf{p}\in V_{d,n}$ by induction on $n$. In the base case $n=d+4$, we have, by Theorem~\ref{thm:equationsVdn}, that $W_{d,d+4}=V_{d,d+4}\cup Y_{d,d+4}$. As $\mathbf{p}\in W_{d,d+4}$ and the points are linearly general, then $\mathbf{p}\in V_{d,d+4}$. Now, we assume the conclusion true for $n-1\geq d+4$, and we prove it for $n$. We need to introduce some notation. For $J\subseteq[n]$, let $\pi_J\colon(\mathbb{P}^d)^n\rightarrow(\mathbb{P}^d)^{|J|}$ be the projection onto the factors labeled by $J$. For a configuration $\mathbf{q}=(q_1,\ldots,q_m)$ of points in $\mathbb{P}^d$, we define the \emph{Veronese envelope} $E_{\mathbf{q}}$ as the union of the quasi-Veronese curves passing through $q_1,\ldots,q_m$.

Let us consider $\pi_{[n-1]}(\mathbf{p})$. This lies in $W_{d,n-1}$ and consists of $n-1$ linearly general points. So, by the inductive step, $\pi_{[n-1]}(\mathbf{p})\in V_{d,n-1}$, in particular they lie on a quasi-Veronese curve $R$. Consider $\pi_{J}(\mathbf{p})$ with $J\in\binom{[n-1]}{d+3}$. In this case, by Castelnuovo's lemma, $E_{\pi_J(\mathbf{p})}$ is a smooth rational normal curve $C_J$. We have the containment
\[
E_{\pi_{[n-1]}(\mathbf{p})}\subseteq\bigcap_{J\in\binom{[n-1]}{d+3}}E_{\pi_J(\mathbf{p})}.
\]
As $E_{\pi_{[n-1]}(\mathbf{p})}\subseteq E_{\pi_J(\mathbf{p})}$, we have that $E_{\pi_{[n-1]}(\mathbf{p})}$ is the smooth rational curve $C_J$. As $J$ was arbitrary, all $E_{\pi_J(\mathbf{p})}$ in the intersection equal $E_{\pi_{[n-1]}(\mathbf{p})}$, so equality holds. Hence, by \cite[Lemma~4.21]{CGMS18}, $\mathbf{p}\in V_{d,n}$ and there is a quasi-Veronese curve $C$ containing $\mathbf{p}$.

Let us prove (2). Suppose $C$ is not a rational normal curve. So, $C=C_1\cup C_2$ with $\deg(C_i)=d_i\geq1$, $d_1+d_2=d$, and $C_i$ spans a linear space of dimension $d_i$. As the points $\mathbf{p}$ are in general linear position, we have at most $d_i+1$ points in $\mathbf{p}$ that lie on $C_i$. But $(d_1+1)+(d_2+1)=d+2<n$, which is a contradiction.
\end{proof}


\section{Generalized von Staudt's theorem: proof of the combinatorial part}
\label{sec:proof-gen-vonStaudt}

We retain the notation of Theorem~\ref{thm:gwt}. For convenience, we relabel $P_1',\ldots,P_{d+1}'$ as $P_{d+2},\ldots,P_{2d+2}$ and $R_1',\ldots,R_{d+1}'$ as $R_{d+2},\ldots,R_{2d+2}$. In this way, the indices $T_1:=\{1,\ldots,d+1\}$ and $T_2:=\{d+2,\ldots,2d+2\}$ describe the first and second simplex respectively.
Thanks to Proposition~\ref{prop:gwt+W}, to prove Theorem~\ref{thm:gwt} it is enough to check that the points $R_{1},\ldots,R_{2d+2}$ satisfy the equations $\psi_{I,J}$ of $W_{2,2d+2}$, which are recalled in \S\,\ref{sec:Vdn-and-equations}. 
To do so, we can assume without loss of generality that the rational normal curve where the points $P_1,\dots,P_{2d+2}$ lie is the standard one. 
 More precisely, we assume that $Q_i=[a_i:b_i]\in\mathbb{P}^1$ are distinct points, $i=1,\ldots,2d+2$, and $P_i=\iota(Q_i)$ where  $\iota\colon\mathbb{P}^1\hookrightarrow\mathbb{P}^d$ is the standard rational normal curve as in \eqref{eq:standard-rational-normal-curve-for-us}. 

The first step is to prove that the brackets $|R_{k_1}\ldots R_{k_{d+1}}|$ decompose as products of two-dimensional brackets $|Q_iQ_j|=
\left|\begin{array}{cc}
	a_i & a_j \\
	b_i & b_j
\end{array}\right|$.  We do this in the following lemma.

\begin{lemma}
\label{lem:bracket-factorization}
Let $K=\{k_1<\ldots<k_{d+1}\}\subseteq[2d+2]$. Define $K_1:=K\cap T_1$ and $K_2:=K\cap T_2$. Then, we have the following equality in the polynomial ring $\kk[a_i,b_i\mid i=1,\dots,2d+2]$:
\[
|R_{k_1}\ldots R_{k_{d+1}}|=(-1)^{\binom{|K_1|}{2}+\binom{|K_2|}{2}}\prod_{\substack{i,j\in K_1\\i< j}}|Q_iQ_j|\prod_{\substack{i,j\in K_2\\i< j}}|Q_iQ_j|\prod_{\substack{i\in T_1\setminus K_1\\j\in T_2\setminus K_2}}|Q_iQ_j|.
\]
\end{lemma}

\begin{proof}
We subdivide the proof into two parts.

\textbf{Part 1.} We first show that the two sides of the equality in the statement differ by a multiplicative constant. We do this by showing that all the degree two factors $|Q_iQ_j|$ of the right-hand side divide the bracket $|R_{k_1}\ldots R_{k_{d+1}}|$. Suppose that $i,j\in K_1$ with $i\neq j$. If $Q_i=Q_j$, then the points $R_{k_i}$ and $R_{k_j}$ coincide, which makes the determinant $|R_{k_1}\ldots R_{k_{d+1}}|$ equal zero. Then, $|Q_iQ_j|$ divides it by the ideal-variety correspondence because the zero set of $|Q_iQ_j|$ is contained in the zero set of $|R_{k_1}\ldots R_{k_{d+1}}|$ and the polynomial $|Q_iQ_j|$ is irreducible \cite[Proposition~1.2]{Har92}. The analogous argument holds for distinct $i,j\in K_2$. Lastly, assume that $i\in T_1\setminus K_1$ and $j\in T_2\setminus K_2$. If $Q_i=Q_j$, then the points $R_{k_1},\ldots,R_{k_{d+1}}$ lie on a common hyperplane, that is, the osculating hyperplane $\pi_i=\pi_j$. This implies that $|R_{k_1}\ldots R_{k_{d+1}}|=0$. So, also in this case, $|Q_iQ_j|$ divides the left-hand side.

Now, the claimed equality holds, up to a constant, because the degree of the polynomial on the right-hand side is
\[
2\left(\binom{|K_1|}{2}+\binom{|K_2|}{2}+(d+1-|K_1|)(d+1-|K_2|)\right),
\]
which, after using the relation $|K_1|+|K_2|=d+1$, equals the degree of the left-hand side: $d(d+1)$.
This shows that
\begin{equation}
\label{eq:equality-up-to-constant}
|R_{k_1}\ldots R_{k_{d+1}}|=c\prod_{\substack{i,j\in K_1\\i< j}}|Q_iQ_j|\prod_{\substack{i,j\in K_2\\i< j}}|Q_iQ_j|\prod_{\substack{i\in T_1\setminus K_1\\j\in T_2\setminus K_2}}|Q_iQ_j|.
\end{equation}
for some constant $c$.

\textbf{Part 2.} We now show that $c=(-1)^{\binom{|K_1|}{2}+\binom{|K_2|}{2}}$. For this purpose, we can choose specific values for $a_i,b_i$. In particular, from now on, we suppose that $b_i=1$ for all $i$. Let $M(T_1,K_1,T_2,K_2)$ be the matrix whose columns are $R_{k_1},\ldots,R_{k_{d+1}}$. There are two cases to consider.

\textbf{The case $K_1=\emptyset$.} This implies that $T_2=K_2$ and \eqref{eq:equality-up-to-constant} becomes
\begin{equation}
\label{eq:case-K1-empty}
|M(T_1,\emptyset,T_2,T_2)|=c\prod_{\substack{i,j\in T_2\\i< j}}|Q_iQ_j|.
\end{equation}
We want to show that $c=(-1)^{\binom{|T_2|}{2}}$. We argue by induction on $d\geq2$. In the base case $d=2$ we have that $T_1=\{1,2,3\}$, $T_2=\{4,5,6\}$, and \eqref{eq:case-K1-empty} becomes:
\begin{displaymath}
\left| \begin{array}{ccc}
a_5a_6&a_4a_6&a_4a_5\\
a_5+a_6&a_4+a_6&a_4+a_5\\
1&1&1
\end{array} \right|=c\left| \begin{array}{cc}
a_4&a_5\\
1&1
\end{array} \right|\left| \begin{array}{cc}
a_4&a_6\\
1&1
\end{array} \right|\left| \begin{array}{cc}
a_5&a_6\\
1&1
\end{array} \right|,
\end{displaymath}
which is satisfied for $c=-1=(-1)^{\binom{|T_2|}{2}}$. Now assume the conclusion true for $d-1\geq2$ and let us prove it for $d$. By letting $a_{d+2}=0$, \eqref{eq:case-K1-empty} becomes
\[
a_{d+3}\ldots a_{2d+2}|M(T_1,\emptyset,T_2\setminus\{d+2\},T_2\setminus\{d+2\})|=c\prod_{\substack{j\in T_2\\d+2< j}}|Q_{d+2}Q_j|\prod_{\substack{i,j\in T_2\\d+2<i< j}}|Q_iQ_j|.
\]
The bracket in the left-hand side of the above equality can be replaced using the inductive hypothesis obtaining
\begin{align*}
a_{d+3}\ldots a_{2d+2}(-1)^{\binom{|T_2|-1}{2}}\prod_{\substack{i,j\in T_2\setminus\{d+2\}\\i< j}}|Q_iQ_j|&=c\prod_{\substack{j\in T_2\\d+2< j}}|Q_{d+2}Q_j|\prod_{\substack{i,j\in T_2\\d+2<i< j}}|Q_iQ_j|,\\
a_{d+3}\ldots a_{2d+2}(-1)^{\binom{|T_2|-1}{2}}&=c\prod_{\substack{j\in T_2\\d+2< j}}|Q_{d+2}Q_j|,\\
a_{d+3}\ldots a_{2d+2}(-1)^{\binom{|T_2|-1}{2}}&=c(-1)^da_{d+3}\cdots a_{2d+2},
\end{align*}
which is satisfied for $c=(-1)^{\binom{|T_2|}{2}}$ because $\binom{|T_2|}{2}=\binom{|T_2|-1}{2}+d$. This proves what we want for $K_1=\emptyset$.

\textbf{Arbitrary $K_1$.} This is done by induction on $|K_1|$. The base case is $K_1=\emptyset$, which was addressed above. So, let us assume the thesis for $|K_1|=p-1\geq0$ and let us prove it for $|K_1|=p$. Since $K_1\neq\emptyset$, then $K_2\neq T_2$. So, we can define
\[
i_0=\min\;K_1~\textrm{and}~j_0=\max\;T_2\setminus K_2.
\]
Let us set $a_{i_0}=a_{j_0}=0$. This implies that the left-hand side of \eqref{eq:equality-up-to-constant} becomes
\[
|R_{k_1}\ldots R_{k_{d+1}}|=|M(T_1,K_1,T_2,K_2)|=\prod_{i\in T_1\setminus\{i_0\}}a_i\cdot|M(T_1',K_1',T_2',K_2)|,
\]
where $T_1'=T_1\setminus\{i_0\}$, $K_1'=K_1\setminus\{i_0\}$, $T_2'=T_2\setminus\{j_0\}$. By using the inductive hypothesis on the matrix $|M(T_1',K_1',T_2',K_2)|$, we argue that
\[
|R_{k_1}\ldots R_{k_{d+1}}|=\prod_{i\in T_1'}a_i\cdot(-1)^{\binom{|K_1'|}{2}+\binom{|K_2|}{2}}\prod_{\substack{i,j\in K_1'\\i< j}}|Q_iQ_j|\prod_{\substack{i,j\in K_2\\i< j}}|Q_iQ_j|\prod_{\substack{i\in T_1'\setminus K_1'\\j\in T_2'\setminus K_2}}|Q_iQ_j|.
\]
Now, let us consider the right-hand side of \eqref{eq:equality-up-to-constant}, which with $a_{i_0}=a_{j_0}=0$ becomes
\begin{align*}
c&\prod_{\substack{i,j\in K_1\\i< j}}|Q_iQ_j|\prod_{\substack{i,j\in K_2\\i< j}}|Q_iQ_j|\prod_{\substack{i\in T_1\setminus K_1\\j\in T_2\setminus K_2}}|Q_iQ_j|\\
=c&\prod_{j\in K_1'}|Q_{i_0}Q_j|\prod_{\substack{i,j\in K_1'\\i< j}}|Q_iQ_j|\prod_{\substack{i,j\in K_2\\i< j}}|Q_iQ_j|\prod_{i\in T_1'\setminus K_1'}|Q_iQ_{j_0}|\prod_{\substack{i\in T_1'\setminus K_1'\\j\in T_2'\setminus K_2}}|Q_iQ_j|\\
=c&(-1)^{|K_1'|}\prod_{j\in K_1'}a_j\prod_{\substack{i,j\in K_1'\\i< j}}|Q_iQ_j|\prod_{\substack{i,j\in K_2\\i< j}}|Q_iQ_j|\prod_{i\in T_1'\setminus K_1'}a_i\prod_{\substack{i\in T_1'\setminus K_1'\\j\in T_2'\setminus K_2}}|Q_iQ_j|\\
=c&(-1)^{|K_1'|}\prod_{j\in T_1'}a_j\prod_{\substack{i,j\in K_1'\\i< j}}|Q_iQ_j|\prod_{\substack{i,j\in K_2\\i< j}}|Q_iQ_j|\prod_{\substack{i\in T_1'\setminus K_1'\\j\in T_2'\setminus K_2}}|Q_iQ_j|.
\end{align*}
By comparing the left-hand side and the right-hand side of \eqref{eq:equality-up-to-constant} after imposing $a_{i_0}=a_{j_0}=0$ and after simplifying the common factors, we obtain that
\[
(-1)^{\binom{|K_1'|}{2}+\binom{|K_2|}{2}}=c(-1)^{|K_1'|}\implies c=(-1)^{\binom{|K_1'|}{2}+|K_1'|+\binom{|K_2|}{2}}=(-1)^{\binom{|K_1|}{2}+\binom{|K_2|}{2}}.\qedhere
\]
\end{proof}

\begin{proposition}
\label{prop:Wdn-equations-are-satisfied}
Let $J\in\binom{[2d+2]}{d+4}$, $I=\{i_1<\ldots<i_6\}\in\binom{J}{6}$, and write $J\setminus I=\{j_1<\ldots<j_{d-2}\}$. Then, the associated equation $\psi_{I,J}=0$ of $W_{d,2d+2}$ is satisfied by the points $R_k$ for $k\in J$.
\end{proposition}

\begin{proof}
Recall that the equation $\psi_{I,J}$ evaluated in $R_k$ for $k\in J$ is equal to $m_{I,J}^{(1)}-m_{I,J}^{(2)}$, where:
\begin{align*}
m_{I,J}^{(1)}&=|i_4i_5i_6j_1\ldots j_{d-2}||i_2i_3i_6j_1\ldots j_{d-2}||i_1i_3i_5j_1\ldots j_{d-2}||i_1i_2i_4j_1\ldots j_{d-2}|,\\
m_{I,J}^{(2)}&=|i_3i_5i_6j_1\ldots j_{d-2}||i_2i_4i_6j_1\ldots j_{d-2}||i_1i_4i_5j_1\ldots j_{d-2}||i_1i_2i_3j_1\ldots j_{d-2}|.
\end{align*}
with the convention that $|i_4i_5i_6j_1\ldots j_{d-2}|=|R_{i_4}R_{i_5}R_{i_6}R_{j_1}\ldots R_{j_{d-2}}|$, and so on. The proof is divided into two steps.

\textbf{Step 1.} We show that $m_{I,J}^{(1)}/m_{I,J}^{(2)}$ is an integral power of $-1$. From Lemma~\ref{lem:bracket-factorization} we have that $m_{I,J}^{(i)}$ decomposes, up to a sign, into the product of appropriate brackets $|Q_iQ_j|$ with $i<j$. We show that the same factors $|Q_iQ_j|$ appear in the decompositions of $m_{I,J}^{(1)}$ and $m_{I,J}^{(2)}$. So, let $|Q_eQ_f|$ be a factor of $m_{I,J}^{(1)}$. By Lemma~\ref{lem:bracket-factorization} we have the following possibilities for $|Q_eQ_f|$:

\begin{enumerate}

\item Suppose $e,f\in T_1$. If $e,f\in J\setminus I$, then the factor $|Q_eQ_f|$ appears with multiplicity $4$ in both $m_{I,J}^{(1)}$ and $m_{I,J}^{(2)}$. If $e,f\in I$, then we have the following possibilities for $\{e,f\}$:
\begin{align*}
\{i_1,i_2\}&,~\{i_1,i_3\},~\{i_1,i_4\},~\{i_1,i_5\},~\{i_2,i_3\},~\{i_2,i_4\},\\
\{i_2,i_6\}&,~\{i_3,i_5\},~\{i_3,i_6\},~\{i_4,i_5\},~\{i_4,i_6\},~\{i_5,i_6\}.
\end{align*}
Then, $|Q_eQ_f|$ appears with multiplicity $1$ in both $m_{I,J}^{(1)}$ and $m_{I,J}^{(2)}$. If $e\in I$ and $f\in J\setminus I$ (and analogously if $f\in I$ and $e\in J\setminus I$), then $|Q_eQ_f|$ appears with multiplicity $2$ in $m_{I,J}^{(1)}$ as any index among $i_1,\ldots,i_6$ appears twice in the triplets $\{i_4,i_5,i_6\},\{i_2,i_3,i_6\},\{i_1,i_3,i_5\},\{i_1,i_2,i_4\}$. This is also true for $m_{I,J}^{(2)}$.

\item The case $e,f\in T_2$ is analogous to the one above.

\item Finally, suppose that $e\in T_1$ and $f\in T_2$. In this case, for $|Q_eQ_f|$ to be a factor of a bracket $|k_1\ldots k_{d+1}|$ of $m_{I,J}^{(1)}$, we must have that $e,f\notin\{k_1,\ldots,k_{d+1}\}$. As $J\setminus I\subseteq\{k_1,\ldots,k_{d+1}\}$, we have that $e,f\in I$. Now, for each possible pair $(e,f)$ with $e,f\in I$ and $e<f$, a direct check reveals that $\{e,f\}\cap\{k_1,\ldots,k_{d+1}\}=\emptyset$ for the same number of brackets in $m_{I,J}^{(1)}$ and $m_{I,J}^{(2)}$. (As an explicit example, for $(e,f)=(i_1,i_6),(i_2,i_5),(i_3,i_4)$, the intersection $\{e,f\}\cap\{k_1,\ldots,k_{d+1}\}$ is never empty for all the brackets in $m_{I,J}^{(1)}$ and $m_{I,J}^{(2)}$.)

\end{enumerate}

\textbf{Step 2.} We now show that $m_{I,J}^{(1)}/m_{I,J}^{(2)}=1$, which concludes the proof. Given a bracket $|i_ai_bi_cj_1\ldots j_{d-2}|$ as it appears in $m_{I,J}^{(1)}$ or $m_{I,J}^{(2)}$, let $q(i_a,i_b,i_c)$ be the number of adjacent transpositions of $i_a,i_b,i_c,j_1,\ldots,j_{d-2}$ which are needed to arrange them in increasing order $H_{abc}:=\{h_1<\ldots<h_{d+1}\}$. A direct check shows that
\begin{align*}
q(i_4,i_5,i_6)+q(i_2,i_3,i_6)+q(i_1,i_3,i_5)+q(i_1,i_2,i_4)&\equiv0~(\mathrm{mod}~2),\\
q(i_3,i_5,i_6)+q(i_2,i_4,i_6)+q(i_1,i_4,i_5)+q(i_1,i_2,i_3)&\equiv0~(\mathrm{mod}~2).
\end{align*}
Therefore, we have that
\begin{align*}
m_{I,J}^{(1)}=|H_{456}||H_{236}||H_{135}||H_{124}|,\\
m_{I,J}^{(2)}=|H_{356}||H_{246}||H_{145}||H_{123}|.
\end{align*}
With the notation of Lemma~\ref{lem:bracket-factorization}, given a bracket $|H_{abc}|$, we define its sign to be
\[
s(H_{abc}):=(-1)^{\binom{|K_1|}{2}+\binom{|K_2|}{2}}.
\]
We now distinguish several cases.

\noindent$\bullet$ $I\cap T_1=\emptyset$ or $I\cap T_1=I$. In these cases we have that the brackets in $m_{I,J}^{(1)},m_{I,J}^{(2)}$ share the same sign. This is due to the fact that for each bracket the number of labels in $T_1$ (resp. $T_2$) is the same.

\noindent$\bullet$ $I\cap T_1=\{i_1\},\{i_1,i_2\},\{i_1,i_2,i_3,i_4\}$, or $\{i_1,i_2,i_3,i_4,i_5\}$. In each one of these cases, we have that
\[
s(H_{456})=s(H_{356}),~
s(H_{236})=s(H_{246}),~
s(H_{135})=s(H_{145}),~
s(H_{124})=s(H_{123}).
\]
This implies that $m_{I,J}^{(1)}/m_{I,J}^{(2)}=1$.

\noindent$\bullet$ $I\cap T_1=\{i_1,i_2,i_3\}$. Define $p=|(J\setminus I)\cap T_1|$. For the brackets in $m_{I,J}^{(1)}$ we have that
\begin{align*}
s(H_{456})&=(-1)^{\binom{p}{2}+\binom{d+1-p}{2}},\\
s(H_{236})&=s(H_{135})=s(H_{134})=(-1)^{\binom{p+2}{2}+\binom{d-1-p}{2}}.
\end{align*}
For the brackets in $m_{I,J}^{(2)}$ we have
\begin{align*}
s(H_{356})&=s(H_{246})=s(H_{145})=(-1)^{\binom{p+1}{2}+\binom{d-p}{2}},\\
s(H_{123})&=(-1)^{\binom{p+3}{2}+\binom{d-2-p}{2}}.
\end{align*}
So, we need to verify that $A\equiv B~\mathrm{mod}~2$, where
\begin{align*}
A&:=\binom{p}{2}+\binom{d+1-p}{2}+3\left(\binom{p+2}{2}+\binom{d-1-p}{2}\right),\\
B&:=3\left(\binom{p+1}{2}+\binom{d-p}{2}\right)+\binom{p+3}{2}+\binom{d-2-p}{2}.
\end{align*}
This is equivalent to checking that, modulo $2$,
\[
\binom{p}{2}+\binom{p+2}{2}+\binom{d-1-p}{2}+\binom{d+1-p}{2}\equiv\binom{p+1}{2}+\binom{p+3}{2}+\binom{d-2-p}{2}+\binom{d-p}{2}.
\]
This follows after observing that, for any integer $\ell$, we have $\binom{\ell}{2}+\binom{\ell+2}{2}\equiv1~\mathrm{mod}~2$.
\end{proof}

\begin{proof}[Proof of Theorem~\ref{thm:gwt}]
This now follows at once after combining Proposition~\ref{prop:gwt+W} and Proposition~\ref{prop:Wdn-equations-are-satisfied}.
\end{proof}

\begin{example}
We give explicit instances of the factorization exhibited by Lemma~\ref{lem:bracket-factorization} and verify Proposition~\ref{prop:Wdn-equations-are-satisfied} directly on an example. Let us consider $d=3$ and the equation $\psi_{I,J}$ for the choice of indexes $J=\{1,\ldots,7\}$ and $I=\{1,\ldots,6\}$. This is given by
\[
\psi_{I,J}= |4567||2367||1357||1247| - |3567||2467||1457||1237|,
\]
with the usual convention that $|4567|$ denotes $|R_4R_5R_6R_7|$, and so on.
By Lemma~\ref{lem:bracket-factorization}, we have the following factorizations
\begin{align*}
|R_4R_5R_6R_7|&=-|Q_5Q_6||Q_5Q_7||Q_6Q_7||Q_1Q_8||Q_2Q_8||Q_3Q_8|\\
|R_2R_3R_6R_7|&=|Q_2Q_3||Q_6Q_7||Q_1Q_5||Q_1Q_8||Q_4Q_5||Q_4Q_8|\\
|R_1R_3R_5R_7|&=|Q_1Q_3||Q_5Q_7||Q_2Q_6||Q_2Q_8||Q_4Q_6||Q_4Q_8|\\
|R_1R_2R_4R_7|&=-|Q_1Q_2||Q_1Q_4||Q_2Q_4||Q_3Q_5||Q_3Q_6||Q_3Q_8|\\
|R_3R_5R_6R_7|&=-|Q_5Q_6||Q_5Q_7||Q_6Q_7||Q_1Q_8||Q_2Q_8||Q_4Q_8|\\
|R_2R_4R_6R_7|&=|Q_2Q_4||Q_6Q_7||Q_1Q_5||Q_1Q_8||Q_3Q_5||Q_3Q_8|\\
|R_1R_4R_5R_7|&=|Q_1Q_4||Q_5Q_7||Q_2Q_6||Q_2Q_8||Q_3Q_6||Q_3Q_8|\\
|R_1R_2R_3R_7|&=-|Q_1Q_2||Q_1Q_3||Q_2Q_3||Q_4Q_5||Q_4Q_6||Q_4Q_8|.
\end{align*}
Now, one can check directly that $\psi_{I,J}=0$, as claimed in Proposition~\ref{prop:Wdn-equations-are-satisfied}.
\end{example}


\subsection*{Acknowledgements}
We would like to thank Ciro Ciliberto for helpful conversations and Irmgard Wolf Tortarolo for providing us with an Italian translation of \cite{Hur82}. We also thank the anonymous referees for the valuable comments and suggestions.


\subsection*{Fundings}
{\small AC and EC acknowledge that this study was carried out within the PRIN 2022 PNRR projects grant P2022J4HRR ``Mathematical Primitives for Post Quantum Digital Signatures'' and grant P2022E2Z4AK ``0-Dimensional Schemes, Tensor Theory, and Applications'' funded by European Union – Next Generation EU  within the PRIN 2022 program (D.D. 104 - 02/02/2022 Ministero dell’Università e della Ricerca). 
AC is also supported by the PRIN2020 grant 2020355B8Y ``Squarefree Gr\"obner degenerations, special varieties and related topics'', by the MUR Excellence Department Project awarded to Dipartimento di Matematica, Università di Genova, CUP D33C23001110001, and by the European Union within the program NextGenerationEU.  Additionally, part of the work was done while visiting the Institute of Mathematics of the University of Barcelona (IMUB). He gratefully appreciates their hospitality during his visit.
LS is supported by the projects ``Programma per Giovani Ricercatori Rita Levi Montalcini'', PRIN2017SSNZAW ``Advances in Moduli Theory and Birational Classification'', PRIN2020KKWT53 ``Curves, Ricci flat Varieties and their Interactions'', and PRIN 2022 ``Moduli Spaces and Birational Geometry'' -- CUP E53D23005790006. The authors are members of the INdAM group GNSAGA.}

	

\end{document}